\documentclass[10pt,a4paper]{amsart}
\usepackage{amssymb,amsmath}
\usepackage[cm,headings]{fullpage}
\usepackage{array}
\usepackage{booktabs}
\usepackage{multirow}
\usepackage{url}
\usepackage{cellspace}
\usepackage{makecell}
\usepackage{multirow, array}
\DeclareMathOperator{\rank}{rank}
\DeclareMathOperator{\myvec}{vec}

\newtheorem{theorem}{Theorem}[section]

\newtheorem{lemma}[theorem]{Lemma}

\newtheorem{proposition}[theorem]{Proposition}

\theoremstyle{definition}

\theoremstyle{definition}

\theoremstyle{definition}

\theoremstyle{definition}
\newtheorem{rem}{Remark}[section]

\title{Constructions of maximum few-distance sets in Euclidean spaces}
\author{Ferenc Sz\"oll\H{o}si and Patric R.J. \"Osterg\aa rd}
\thanks{\today, Preprint. This research was supported in part by the Academy of Finland, Grant \#289002}
\address{F. Sz. and P.R.J. \"O: Department of Communications and Networking, Aalto University School of Electrical Engineering, P.O. Box 15400, 00076 Aalto, Finland}
\email{szoferi@gmail.com, patric.ostergard@aalto.fi}

\begin{document}
\begin{abstract}
A finite set of distinct vectors $\mathcal{X}$ in the $d$-dimensional Euclidean space $\mathbb{R}^d$ is called an $s$-distance set if the set of mutual distances between distinct elements of $\mathcal{X}$ has cardinality $s$. In this paper we present a combined approach of isomorph-free exhaustive generation of graphs and Gr\"obner basis computation to classify the largest $3$-distance sets in $\mathbb{R}^4$, the largest $4$-distance sets in $\mathbb{R}^3$, and the largest $6$-distance sets in $\mathbb{R}^2$. We also construct new examples of large $s$-distance sets for $d\leq 8$ and $s\leq 6$, and independently verify several earlier results from the literature.
\end{abstract}
\maketitle
\section{Introduction}
Let $d\geq 1$ be an integer, and let $\mathbb{R}^d$ denote the $d$-dimensional Euclidean space equipped with the standard inner product $\left\langle .,.\right\rangle$ and norm induced metric $\mu(.,.)$. A set of $n$ distinct vectors $\mathcal{X}:=\{v_i\colon i\in\{1,\dots,n\}\}\subset\mathbb{R}^d$ forms an $s$-distance set, if the set of mutual distances $A(\mathcal{X}):=\{\mu(v_i,v_j)\colon i<j; i,j\in\{1,\dots,n\}\}$ has cardinality $s$. If the parameter $s$ is not specified, then, following the terminology of \cite{blok}, we refer to these objects as few-distance sets. A few-distance set $\mathcal{X}$ is called spherical, if $\left\langle v_i,v_i\right\rangle=1$ for every $i\in\{1,\dots,n\}$. The problem of determining the maximum cardinality of $\mathcal{X}$ for given $d$ and $s$ is a long-standing open problem \cite{ME}, \cite{ES}, \cite{LS}. Currently there is a renewed interest in few-distance sets due to recent breakthrough results obtained via the polynomial method \cite{BY}, \cite{M}, and because of emerging engineering applications related to frame theory \cite{frame1}, \cite{EDG}, \cite{SW}. 

The case $s=2$ is the most studied, where equiangular lines appear as a special case \cite{FSZequi}, \cite{SZFseidel}. In particular, the maximum cardinality of $2$-distance sets in $\mathbb{R}^d$ have already been determined for $d\leq 8$, see \cite{croft}, \cite{kelly}, \cite{L}. For $s\leq 6$ the maximum cardinality of planar $s$-distance sets are also known \cite{EF}, \cite{xx}. However, in the case $d\geq 3$ and $s\geq 3$ the only known result is that the set of $12$ vertices of the icosahedron is the unique maximum $3$-distance set in $\mathbb{R}^3$, see \cite{Sico}.

We mention two fundamental properties of maximum few-distance sets: first, the asymptotic growth of their size is well understood. In particular, if $\mathcal{X}$ is a maximum $s$-distance set in $\mathbb{R}^d$, then we have
\begin{equation}\label{main_eq1}
|\mathcal{X}|\leq\binom{d+s}{s};\quad \text{ and, for $d\geq 2s-1$, }\quad |\mathcal{X}|\geq\binom{d+1}{s}.
\end{equation}
The upper bound is from \cite{BBS}, \cite{blok}, and the folklore constructive lower bound is mentioned in, e.g. \cite{EXC}, \cite{NM}. Somewhat improved upper bounds apply to the spherical case \cite{NM}. Secondly, there is a rather strong number theoretic condition constraining the elements of $A(\mathcal{X})$, which holds for sufficiently large $d$ in terms of $s$, see \cite{LRS}, \cite{N}. One reason why studying small dimensional maximum few-distance sets is challenging is the lack of sufficient understanding of $A(\mathcal{X})$.

In this paper we construct new $s$-distance sets in $\mathbb{R}^d$ for small parameter values, and prove their optimality in certain cases. This complements recent efforts \cite{BY}, \cite{NM} aimed to strengthening the upper bounds on $|\mathcal{X}|$. The discovery of a $16$-element $3$-distance set in $\mathbb{R}^4$ and the proof of its uniqueness and optimality is one of the main contributions of this paper (Theorem~\ref{mainthmxxx}). In addition, we settle \cite[Conjecture~2]{EF} by fully classifying all planar $6$-distance sets (Theorem~\ref{erdosthm}), extending previous partial results on this problem \cite{xx}. We also investigate spherical few-distance sets and discover some connections to near-resolvable designs and generalized conference matrices (Remarks~\ref{rem1} and \ref{rem2}). This correspondence opens up various new avenues to be explored in search for large few-distance sets in the future.

The proofs are computational, and they involve both classical graph generation techniques \cite{MCK}, \cite{read}, and elements of computational commutative algebra \cite{GBBOOK}, in particular: Gr\"obner basis calculations. The graph generation is required to gather combinatorial information on the structure of the putative few-distance set $\mathcal{X}$, while studying its algebraic properties is necessary to gain information on $A(\mathcal{X})$ and control the size of the ambient dimension $d$. This approach builds upon, and considerably extends the earlier work \cite{L}, where maximum $2$-distance sets were studied by means of computers. We remark that the use of computational commutative algebra recently lead to the resolution of several challenging problems in metric geometry, such as finding a unit-distance planar embedding of the Heawood graph \cite{HEA}, determining optimal packings in real projective spaces \cite{DM}, or showing the nonexistence of certain complex equiangular tight frames \cite{SZF}.

The outline of this paper is as follows: in Section~2 we describe constructions of few-distance sets, and establish lower bounds on the size of maximum few-distance sets for small parameter values, see Table~\ref{TableBds}. In Section~3 we set up a computational framework for generating and classifying spherical few-distance sets in $\mathbb{R}^d$. Then, in Section~4 we slightly modify this approach in order to deal with general (i.e., not necessarily spherical) few-distance sets.

In the following Table~\ref{TableBds} we tabulate the exact values of, and best known lower bounds on the cardinality of $n$-element $s$-distance sets in $\mathbb{R}^d$ for $d\leq 8$ and $s\leq 6$, which reflects the new results obtained in this paper. We will comment on the entries in the next section.

\begin{table}[htbp]%
\tiny
\begin{tabular}{cccccccc}
\hline
${}_{s}\mkern-6mu\setminus\mkern-6mu{}^{d}$ &  2 & 3 & 4 & 5 & 6 & 7 & 8\\
\hline
2 &  5 &   6 & 10 &  16 &   27 & 29   & 45\\
3 &  7 &  12 & 16 & 24- &  40- & 65-  & 121-\\
4 &  9 &  13 & 25-& 41- &  73- & 127- & 241-\\
5 & 12 & 20- & 35-& 66- & 112- & 168- & 252-\\
6 & 13 & 21- & 40-& 96- & 141- & 281- & 505-\\
\hline
\end{tabular}
\caption{Lower bounds on and exact values of $n$-element $s$-distance sets in $\mathbb{R}^d$}
\label{TableBds}
\normalsize
\end{table}
\section{Constructions of few-distance sets}
In this section we describe a computer-aided construction of $s$-distance sets $\mathcal{X}\subset\mathbb{R}^d$. The construction is based on the following ansatz: it is assumed that the set of distances $A(\mathcal{X})$ is a subset of those within an $s$-dimensional unit cube, and in addition a $(d-1)$-dimensional unit simplex is a constituent of the configuration. The construction, which is summarized in the next theorem, results in maximum $s$-distance sets in several small dimensional cases. For a graph $\Gamma$ we denote by $\omega(\Gamma)$ the size of its maximum clique, that is, pairwise adjacent vertices.
\begin{theorem}\label{thm41}
Let $d\geq 1, s\geq 2$ be integers, let $A:=\{\sqrt{2i}\colon i\in\{1,\dots,s\}\}$, let $\mathcal{W}:=\{(w_1,\dots,w_d)^T\in\mathbb{R}^d\colon w_i\in A, i\in\{1,\dots,d\}\}$, and for $w\in\mathcal{W}$ let $\mathcal{V}_w$ be the set of solution vectors $v:=(v_1,\dots,v_d)^T\in\mathbb{R}^d$ to the following system of $d$ equations in $d$ variables:
\begin{equation}\label{eq_main_fish}
\sum_{i=1}^d\left((w_1^2-w_i^2)/2+v_1\right)^2=w_1^2+2v_1-1, \qquad v_j=(w_1^2-w_j^2)/2+v_1,\quad j\in\{2,\dots, d\}.
\end{equation}
Let $\Gamma$ be a graph on $\sum_{w\in\mathcal{W}} |\mathcal{V}_w|$ vertices, where two nodes, representing the vectors $x,y\in\cup_{w\in\mathcal{W}}\mathcal{V}_w$, are adjacent if and only if $\mu(x,y)\in A$. Then for some $A(\mathcal{X})\subseteq A$ there exists an $|A(\mathcal{X})|$-distance set $\mathcal{X}$ in $\mathbb{R}^d$ with $d+\omega(\Gamma)$ elements.
\end{theorem}
\begin{proof}
Let $B_i\in\mathbb{R}^d$ denote the canonical basis vector whose $i$th coordinate equals $1$, and its other coordinates are $0$. Note that for distinct $i,j\in\{1,\dots, d\}$, we have $\mu(B_i,B_j)=\sqrt{2}$, which is the smallest element of $A$. For a fixed $w\in \mathcal{W}$ the system of equations \eqref{eq_main_fish} is easily seen to be equivalent to the system of equations $\mu(v,B_i)=w_i$, $i\in\{1,\dots,d\}$. Since these equations are quadratic in $v_1$, which uniquely determines $v_i$ for $i\in\{2,\dots, d\}$, the graph $\Gamma$ is finite, and a clique in it represents a subset of solution vectors whose mutual distances belong to $A$. Since $\{B_1,\dots, B_d\}$ is disjoint from $\cup_{w\in\mathcal{W}}\mathcal{V}_w$, the result follows.
\end{proof}
We use Theorem~\ref{thm41} to establish lower bounds on the maximum cardinality of $s$-distance sets in $\mathbb{R}^d$ for small parameter values. Solving the quadratic equations \eqref{eq_main_fish} is straightforward as well as setting up the resulting compatibility graph $\Gamma$ on at most $2s^d$ vertices. The clique search was carried out with the cliquer software \cite{CLQ} which is based on the algorithm described in \cite{PAT2}. In the following Table~\ref{fishingtable} we display our results in the form $|
\Gamma|/\omega(\Gamma)$. Entries marked by an asterisk indicate a lower bound on $\omega(\Gamma)$. Note that the row and column headings shown indicate the input parameters used, both of which are upper bounds on the parameters of the resulting few-distance sets. In particular, entries shown in row $s$ correspond to $t$-distances sets with some $t\leq s$. For further bounds on planar few-distance sets, see \cite{EF}.

\begin{table}[htbp]%
\tiny
\begin{tabular}{rrrrrrrr}
\hline
${}_{s}\mkern-6mu\setminus\mkern-6mu{}^{d}$ &  2 & 3 & 4 & 5 & 6 & 7 & 8\\
\hline
2 &   8/2 &   16/3 &   32/6 &   64/11 &   128/21 &  256/22 & 456/37\\
3 &  18/2 &   51/5 & 130/12 &  306/19 &  686/34 & 1497/58*   &    \\
4 &  30/5 &  95/10 & 272/21 &  738/36 & 1916/67 &  & \\
5 &  44/5 & 163/13 & 542/31 & 1650/61 & 4698/106* &  & \\
6 &  58/5 & 237/13 & 876/36 & 2982/91 &  &  & \\
\hline
\end{tabular}
\caption{Compatibility graph sizes and clique sizes}
\label{fishingtable}
\normalsize
\end{table}
\begin{rem}
For $s=2$, $d=9$ the graph $\Gamma$, coming from Theorem~\ref{thm41} is on $442$ vertices with $\omega(\Gamma)=36$. Therefore it is not possible to immediately improve with this construction on the cardinality of known $2$-distance sets in $\mathbb{R}^9$. We also note here that starting from $d\geq 13$ the chosen set of distances might no longer be optimal for $2$-distance sets, and instead $A:=\{\sqrt{2},\sqrt{3}\}$ should be preferred \cite{LRS}.
\end{rem}
\begin{rem}
There are a number of ways to generalize Theorem~\ref{thm41}: not only the set of distances, and the shape of the constituent vectors are subject to our choice, but also the relative size of the constituent vectors. In particular, assuming that the $(d-1)$-simplex is formed by those vectors whose mutual distance is the largest element of $A(\mathcal{X})$ could, and it indeed does, result in different compatibility graphs having different clique sizes than of those reported in Table~\ref{fishingtable}. We have not yet explored any of these paths.
\end{rem}
Next we comment on Table~\ref{TableBds} by comparing it with Table~\ref{fishingtable}. As a result of our choice for $A(\mathcal{X})$ the implied bounds on the sizes of few-distance sets coming from Theorem~\ref{thm41} (see Table~\ref{fishingtable}) are arguably weak for $d\in\{2,3\}$ as we are indeed missing most of the regular convex polygons as well as the icosahedron and the dodecahedron. These cases were treated in \cite{ES}, \cite{EF}, \cite{xx}. On the other hand, for $d\leq 8$ we were able to rediscover the maximum $2$-distance sets reported in \cite{L}. The case $d=3$, $s=3$ is the icosahedron, whose optimality was shown in \cite{Sico}, and we will see later in Section~4 that the $3$-distance sets in $\mathbb{R}^4$ and the $4$ distance sets in $\mathbb{R}^3$ constructed by Theorem~\ref{thm41} are also maximum few-distance sets. For $d=5$ the reported lower bounds seem to be new. The $4_{21}$ polytope \cite{COX} and its center point corresponds to the case $d=8$, $s=4$, and its various subconfigurations give rise to bounds for $d\in\{6,7,8\}$ and $s\in\{3,4\}$. In some of these cases our construction, which not necessarily results in spherical sets, is somewhat better. Finally, for $s\in\{5,6\}$ and $d$ large enough the bounds, which are rather weak, follow from configurations formed by the vertices of variously truncated $d$-simplices. We mention two of these constructions below. We denote by $\mathfrak{S}_i$ the symmetric group on $i$ elements.
\begin{lemma}\label{lemmaA}
Let $d\geq 4$ be an integer. Then $\mathcal{X}:=\{\sigma(1,1,-1,0,\dots,0)^T\colon\sigma\in\mathfrak{S}_{d+1}\}$ forms a $(d-1)\binom{d+1}{2}$-element $5$-distance set in $\mathbb{R}^d$.
\end{lemma}
\begin{proof}
Tedious case-by-case analysis. To see that these vectors are in $\mathbb{R}^d$, observe that after proper rescaling and translating they are on a hyperplane perpendicular to $(1,1,\dots,1)^T\in\mathbb{R}^{d+1}$.
\end{proof}
\begin{lemma}\label{lemmaB}
Let $d\geq 5$ be an integer. Then $\mathcal{X}:=\{\sigma(1,1,1,-1,0,\dots,0)^T\colon\sigma\in\mathfrak{S}_{d+1}\}\cup\{(0,0,\dots,0)^T\in\mathbb{R}^{d+1}\}$ forms a $1+(d-2)\binom{d+1}{3}$-element $6$-distance set in $\mathbb{R}^d$.
\end{lemma}
\begin{proof}
Similar to the proof of Lemma~\ref{lemmaA}.
\end{proof}
Since these two constructions result in essentially spherical few-distance sets, it is plausible that they can be significantly improved. Indeed, equation \eqref{main_eq1} shows that for $d\geq 15$ the resulting sets are far from being optimal. This concludes the discussion of the numbers shown in Table~\ref{fishingtable} and the implied bounds shown in Table~\ref{TableBds}. In the next sections we set up a framework to show the optimality of some of the few-distance sets discovered.
\section{Spherical few-distance sets and their Gram matrices}\label{section2}
In this section we discuss spherical few-distance sets, that is, it is further assumed that the elements of $\mathcal{X}\subset\mathbb{R}^d$ are of unit length. Given a set of mutual distances, deciding whether there is a corresponding spherical configuration is basically a test of positive semidefiniteness. The following result is well-known, see e.g.~\cite{NEU} for an equivalent characterization. We include this result along with its proof for completeness, and for future reference.
\begin{lemma}\label{lemma1}
Let $d\geq1,n\geq2$ be integers. There exists $n$ distinct unit vectors $v_1,\dots,v_n$ in $\mathbb{R}^d$ with mutual distances $\mu(v_i,v_j)$, $i,j\in\{1,\dots,n\}$, if and only if the matrix $G:=[1-\mu(v_i,v_j)^2/2]_{i,j=1}^{n}$ is positive semidefinite, $\rank G\leq d$, and $G_{ij}<1$ for every $i<j$ with $i,j\in\{1,\dots,n\}$.
\end{lemma}
\begin{proof}
Assume that we have a set of $n$ distinct unit vectors $v_1$, $\dots$, $v_n$ with pairwise distances $\mu(v_i,v_j)$, with $i,j\in\{1,\dots, n\}$. By simple algebra we obtain $1-\mu(v_i,v_j)^2/2=\left\langle v_i,v_j\right\rangle$ and therefore $G$ is a Gram matrix. It is well-known that every Gram matrix is positive semidefinite, and its rank is the maximum number of linearly independent vectors amongst $v_i$, $i\in\{1,\dots,n\}$. Therefore $\rank G\leq d$, as claimed. Since the vectors are distinct, we have $\mu(v_i,v_j)>0$ and consequently $G_{ij}<1$ for every $i<j$ with $i,j\in\{1,\dots, n\}$ as claimed.

Conversely, let $G$ be an $n\times n$ positive semidefinite matrix with $G_{ii}=1$ and $G_{ij}<1$ for every $i<j$ with $i,j\in\{1,\dots, n\}$ with $\rank G\leq d$. Then one may define a set of $\binom{n}{2}$ positive real numbers $\mu_{ij}:=\sqrt{2-2G_{ij}}$ for $i<j$ with $i,j\in\{1,\dots,n\}$, and reconstruct $n$ unit vectors with exactly these distances by the Cholesky decomposition with complete pivoting. The procedure will result in a $\rank G\times n$ matrix $V$ such that $V^TV=G$. Since $\mu_{ij}>0$, the column vectors of $V$ are distinct, and they will form the unit vectors $v_i$, $i\in\{1,\dots,n\}$ having the prescribed distances $\mu(v_i,v_j)=\mu_{ij}$ for every $i<j$ with $i,j\in\{1,\dots, n\}$. This reconstruction is unique up to isometry \cite{NJH}.
\end{proof}
From Lemma~\ref{lemma1} it is clear that the relative position of the matrix elements within a Gram matrix encode the combinatorial structure of the represented distance set, their algebraic value encode the distances, while the rank of the matrix encodes the ambient dimension. Therefore our aim is to generate, with computer-aided methods, the Gram matrices. On the one hand, we need to generate all symmetric matrices with constant diagonal $1$ satisfying a combinatorial constraint, namely that their off-diagonal entries assume $s$ distinct values only, say $x_1$, $x_2$, $\dots$, $x_s$. We call these objects candidate Gram matrices. On the other hand, we are facing with the following algebraic problem: for a given candidate Gram matrix $G(x_1,\dots,x_s)$ determine all such real values $\alpha_i$, for which $-1\leq \alpha_i<1$ holds for every $i\in\{1,\dots,s\}$ and $G(\alpha_1,\dots,\alpha_s)$ is positive semidefinite of rank at most $d$, that is, a Gram matrix. We discuss these two tasks in the following sections.
\subsection{Generating candidate Gram matrices}\label{section21}
We generate the candidate Gram matrices $G$ of order $n$ with at most $s$ distinct off-diagonal entries with orderly generation \cite[Section~4.2.2]{KO}, \cite{read}. Since this is a routine task we only outline the basic ideas. Two candidate Gram matrices $G_1(x_1,\dots,x_s)$ and $G_2(x_1,\dots,x_s)$ are called equivalent, if $G_2(x_1,\dots,x_s)=PG_1(x_{\sigma(1)},\dots,x_{\sigma(s)})P^T$ for some permutation matrix $P$ and for some permutation $\sigma\in\mathfrak{S}_s$. Clearly, it is enough to generate these matrices up to equivalence. The generation starts from the $1\times 1$ matrix $\left[\begin{array}{c}1\end{array}\right]$, and then a new row and column with the elements $x_1$, $\dots$, $x_s$ is inductively appended to it in all possible ways, maintaining symmetry and entries $1$ on the main diagonal. However, only those $i\times i$ matrices ($i\in\{2,\dots,n\}$) are kept whose vectorization $\myvec G:=[G_{2,1},G_{3,1},G_{3,2},\dots,G_{i,1},G_{i,2},\dots,G_{i,i-1}]$ is the lexicographically smallest vector in the set $\{\myvec PG(x_{\sigma(1)},\dots,x_{\sigma(s)})P^T\colon P\in\mathfrak{S}_i, \sigma\in\mathfrak{S}_s\}$. An alternative approach completing this task is to employ canonical augmentation \cite{MCK}. We tabulate the number of generated matrices in Table~\ref{TableGraphCount}.

\tiny\begin{table}[htbp]%
\begin{tabular}{rrrrrrrrrr}
${}_s\setminus ^n$ & 2 & 3 &  4 &    5 &      6 &      7 &    8 &      9\\
\hline
                 2 & 1 & 2 &  6 &   18 &     78 &    522 & 6178 & 137352\\
                 3 & 1 & 3 & 15 &  142 &   4300 & 384199 & 98654374\\
                 4 & 1 & 3 & 22 &  513 &  67685 & 37205801\\
                 5 & 1 & 3 & 24 &  956 & 370438 &\\
							   6 & 1 & 3 & 25 & 1205 &\\%
\end{tabular}
\caption{The count for $n\times n$ candidate Gram matrices with at most $s$ distinct off-diagonal entries.}
\label{TableGraphCount}%
\end{table}\normalsize%

We remark that this task can be thought as what is essentially a graph generating problem: the candidate Gram matrices correspond to distinct colorings of the edges of the complete graph on $n$ vertices with at most $s$ colors, up to permutation of the colors. There are techniques, such as the Power Group Enumeration Theorem, which can be used to enumerate these objects without actually being generated. This can be used to independently verify the entries of Table~\ref{TableGraphCount}, see \cite[Chapter~6]{HP}.

Since the number of candidate Gram matrices up to equivalence grows very rapidly, it is important to discard those which cannot correspond to a desired $n$-element $s$-distance set in $\mathbb{R}^d$ during early stages of the search. In the next subsection we discuss a strategy for doing this.
\subsection{Discarding candidate Gram matrices}
During the generation of candidate Gram matrices, we discard those which cannot have appropriate rank. The key observation is that the rank of the Gram matrices is a hereditary property in the following sense.
\begin{lemma}\label{lemma2}
Let $s\geq 1$, $n\geq 2$ be integers, let $x_1$, $\dots$, $x_s$ be indeterminates, let $G(x_1,\dots,x_s)$ be an $n\times n$ candidate Gram matrix, and let $H(x_1,\dots,x_s)$ an $(n-1)\times (n-1)$ submatrix of $G(x_1,\dots,x_s)$. Then, for every $s$ complex numbers $\alpha_1$, $\dots$, $\alpha_s$, we have $\rank H(\alpha_1,\dots,\alpha_s)\leq \rank G(\alpha_1,\dots,\alpha_s)$.
\end{lemma}
\begin{proof}
Assume that $\rank G(\alpha_1,\dots,\alpha_s)=d$ for some positive integer $d$. If $d\in\{n-1,n\}$ then the statement is obvious, therefore we may assume that $d\leq n-2$. Then necessarily every $(d+1)\times (d+1)$ submatrix of $G(\alpha_1,\dots,\alpha_s)$ has vanishing determinant. In particular every $(d+1)\times (d+1)$  submatrix of $H(x_1,\dots,x_s)$ has vanishing determinant. But then $\rank H(\alpha_1,\dots,\alpha_s)\leq d=\rank G(\alpha_1,\dots,\alpha_s)$ as claimed.
\end{proof}
Therefore one should understand first what are the small candidate Gram matrices with appropriate rank, and then build up the larger ones from those. In the next result we state how to control the rank of candidate Gram matrices.
\begin{proposition}\label{prop21}
Let $d,s\geq 1$, $n\geq d+1$ be integers, let $x_1$, $\dots$, $x_s$ be indeterminates, and let $G(x_1,\dots,x_s)$ be an $n\times n$ candidate Gram matrix. Let $\mathcal{M}$ denote the set of all $(d+1)\times (d+1)$ submatrices of $G(x_1,\dots,x_s)$. There exists $s$ distinct complex numbers $\alpha_1$, $\dots$, $\alpha_s$, each different from $1$, so that $\rank G(\alpha_1,\dots,\alpha_s)\leq d$ if and only if the following system of $\binom{n}{d+1}^2+1$ polynomial equations in $s+1$ variables
\begin{equation}\label{main_eq2}
    \mathrm{det}M(x_1,\dots,x_s) = 0,\quad \text{ for all } M(x_1,\dots,x_s)\in\mathcal{M},\qquad 1+u\prod_{i=1}^s(x_i-1)\prod_{1\leq j<k\leq s}(x_j-x_k)=0
\end{equation}
has a complex solution.
\end{proposition}
\begin{proof}
The rank condition on $G$ is equivalent to the condition on vanishing minors. Moreover, the last equation featuring the auxiliary variable $u$ ensures that the values $\alpha_i$, $i\in\{1,\dots,s\}$, are necessarily distinct, and they all different from $1$.
\end{proof}
Proposition~\ref{prop21} is just a necessary condition ensuring that $G(\alpha_1,\dots,\alpha_s)$ is of correct rank for certain parameter values. Indeed, a solution to \eqref{main_eq2} is not necessarily real, and the positive semidefiniteness of $G$ is not guaranteed. Nevertheless, it is a powerful condition eliminating a significant proportion of candidate Gram matrices which cannot correspond to a spherical $s$-distance set for a given $d$. We study the system of equations \eqref{main_eq2} in the following way: we use various computer algebra systems (such as e.g. \cite{cocoa}) to compute an exact degree reverse lexicographic reduced Gr\"obner basis (with respect to the variable ordering $x_1 > \dots > x_s > u$). If this happens to be $\{1\}$, then there are no complex solutions to \eqref{main_eq2}, and the matrix $G(x_1,\dots,x_s)$ can be discarded, as it cannot have appropriate rank. On the other hand, when the Gr\"obner basis is anything else than $\{1\}$, then regardless of what information it conveys, we keep the matrix. For background on computational commutative algebra, we refer the reader to \cite[Chapter~5]{GBBOOK}.

\subsection{The search.}\label{section23} For a fixed $d$ and $s$ the search for an $s$-distance set $\mathcal{X}$ proceeds in essentially the same way as described in \cite[Section~7]{L}. First we generate the $(d+1)\times (d+1)$ candidate Gram matrices with at most $s$ distinct off-diagonal entries, as described in Section~\ref{section21}. Then we filter these matrices with the aid of Proposition~\ref{prop21}: those for which the system of equations \eqref{main_eq2} has no solutions are discarded, the others are retained in a set $L_{d+1}$. Now given a set $L_i$ (for some $i\geq d+1$), we inductively build up larger matrices via the orderly generating algorithm, and we test whether (a) each of their $i\times i$ principal submatrices belong to the set $L_i$ up to equivalence (see Lemma~\ref{lemma2}); and (b) whether they themselves satisfy the system of equations \eqref{main_eq2} of Proposition~\ref{prop21}. The matrices surviving both tests then form the set $L_{i+1}$. We continue doing this as long as bigger matrices are kept being discovered, but only up to the upper bound given by equation \eqref{main_eq1}. Once the search concludes, we inspect the largest candidate Gram matrices found and investigate for what parameter values they are positive semidefinite.
\begin{rem}
A given candidate Gram matrix could actually correspond to multiple nonisometric configurations (which is indeed the case for the two pentagonal pyramids in $\mathbb{R}^3$, see \cite{ES}). However, once the distances are specified, there is an essentially unique way to reconstruct $\mathcal{X}$, as guaranteed by the Cholesky decomposition. See also \cite{NEU}.
\end{rem}
\subsection{Results}\label{sect34} The search described above yielded the following new classification results.
\begin{theorem}\label{tsphdim3dis4}
The largest cardinality of a spherical $4$-distance set in $\mathbb{R}^3$ is $12$. There are exactly two configurations realizing this up to isometry: the vertices of the cuboctahedron, and the vertices of the truncated tetrahedron.
\end{theorem}
Since the proof of Theorem~\ref{tsphdim3dis4} is a result of computer calculations, we could only offer here what is merely a discussion of our findings. First we note that the vertices of the cuboctahedron can be obtained as $\{\sigma(0,1,1,2)^T\colon \sigma\in\mathfrak{S}_4\}$, whereas vertices of the truncated tetrahedron can be obtained as $\{\sigma(0,0,1,2)^T\colon \sigma\in\mathfrak{S}_4\}$. It is easy to see that both of these are $12$-element $4$-distance sets. Moreover, by appropriate translation and scaling one can make these vectors orthogonal to $(1,1,1,1)^T\in\mathbb{R}^4$, and therefore these sets genuinely belong to $\mathbb{R}^3$. 
 
The search resulted in the following largest candidate Gram matrices, denoted by $S_{12A}(u,v,w)$, $S_{12B}(u,v,w,x)$, and $S_{12C}(u,v,w,x)$, respectively:
\setlength{\arraycolsep}{2.5pt}%
\begin{equation}\label{icosahere}\left[
\begin{smallmatrix}
 1 & u & u & u & u & u & v & v & v & v & v & w \\
 u & 1 & u & u & v & v & u & u & v & v & w & v \\
 u & u & 1 & v & u & v & u & v & u & w & v & v \\
 u & u & v & 1 & v & u & v & u & w & u & v & v \\
 u & v & u & v & 1 & u & v & w & u & v & u & v \\
 u & v & v & u & u & 1 & w & v & v & u & u & v \\
 v & u & u & v & v & w & 1 & u & u & v & v & u \\
 v & u & v & u & w & v & u & 1 & v & u & v & u \\
 v & v & u & w & u & v & u & v & 1 & v & u & u \\
 v & v & w & u & v & u & v & u & v & 1 & u & u \\
 v & w & v & v & u & u & v & v & u & u & 1 & u \\
 w & v & v & v & v & v & u & u & u & u & u & 1 \\
\end{smallmatrix}
\right],\quad \left[
\begin{smallmatrix}
 1 & u & u & u & v & v & v & v & w & w & x & x \\
 u & 1 & u & v & u & v & w & x & v & x & v & w \\
 u & u & 1 & v & v & u & x & w & x & v & w & v \\
 u & v & v & 1 & w & w & u & u & v & v & x & x \\
 v & u & v & w & 1 & w & v & x & u & x & u & v \\
 v & v & u & w & w & 1 & x & v & x & u & v & u \\
 v & w & x & u & v & x & 1 & u & u & v & v & w \\
 v & x & w & u & x & v & u & 1 & v & u & w & v \\
 w & v & x & v & u & x & u & v & 1 & w & u & v \\
 w & x & v & v & x & u & v & u & w & 1 & v & u \\
 x & v & w & x & u & v & v & w & u & v & 1 & u \\
 x & w & v & x & v & u & w & v & v & u & u & 1 \\
\end{smallmatrix}
\right],\quad \left[
\begin{smallmatrix}
 1 & u & u & u & u & v & v & w & w & w & w & x \\
 u & 1 & u & v & w & u & w & u & v & w & x & w \\
 u & u & 1 & w & v & w & u & v & u & x & w & w \\
 u & v & w & 1 & u & u & w & w & x & u & v & w \\
 u & w & v & u & 1 & w & u & x & w & v & u & w \\
 v & u & w & u & w & 1 & x & u & w & u & w & v \\
 v & w & u & w & u & x & 1 & w & u & w & u & v \\
 w & u & v & w & x & u & w & 1 & u & v & w & u \\
 w & v & u & x & w & w & u & u & 1 & w & v & u \\
 w & w & x & u & v & u & w & v & w & 1 & u & u \\
 w & x & w & v & u & w & u & w & v & u & 1 & u \\
 x & w & w & w & w & v & v & u & u & u & u & 1 \\
\end{smallmatrix}
\right].
\end{equation}
It is interesting to note that the $4$-distance sets given in Theorem~\ref{tsphdim3dis4} have no more vertices than the icosahedron, which is the maximal $3$-distance set in $\mathbb{R}^3$ and represented by the permutation equivalent matrices $S_{12A}(\pm1/\sqrt{5},\mp1/\sqrt{5},-1)$. This is a ``shadow'' solution featuring fewer than $4$ distinct off-diagonal entries, and therefore it is discarded.\footnote{Several authors, see e.g.~\cite{M}, prefer to call a set $\mathcal{X}$ an $s$-distance set if $|A(\mathcal{X})|\leq s$.} The matrix $S_{12B}(u,v,w,x)$ corresponds to the truncated tetrahedron. By solving the rank-equations coming from Proposition~\ref{prop21} we find that $u=7/11$, $v=-1/11$, $w=-5/11$, and $x=-9/11$ is the only solution for which $S_{12B}(u,v,w,x)$ is positive semidefinite of rank $3$. Finally, for the third case we have $u=\pm1/2$, $v=0$, $w=-u$, $x=-1$. The two algebraic solutions $S_{12C}(1/2,0,-1/2,-1)$ and $S_{12C}(-1/2,0,1/2,-1)$ are permutation equivalent, thus represent isometric configurations. These matrices correspond to the vertices of the cuboctahedron. In Table~\ref{TableSphDim3Dis4} we display the number of intermediate configurations found: in column $i\in\{4,\dots,13\}$ the first row displays the number of $i\times i$ matrices found during the augmentation step, and the second row displays the number of matrices surviving the rank condition \eqref{main_eq2}. Since for $i=8$ this second round of filtering was not effective, we stopped performing it on larger matrices. We remark that the intermediate objects are not necessarily correspond to positive semidefinite matrices and hence they might not have any geometrical meaning in the Euclidean space. This concludes the discussion of Theorem~\ref{tsphdim3dis4}.

\begin{table}[htbp]%
\tiny
\begin{tabular}{ccccccccccc}
\hline
$n$                                                   &  4 &   5 &     6 &  7 &    8 &  9 & 10 & 11 & 12 & 13\\
\hline
\#Candidate Gram matrices                             & 22 & 513 & 36994 & 404 & 179 & 67 & 27 &  3 &  3 &  0\\
\#Candidate Gram matrices satisfying \eqref{main_eq2} & 22 & 434 &  1283 & 383 & 179 & \\
\hline
\end{tabular}
\caption{Spherical $n$-element $4$-distance sets in $\mathbb{R}^3$}
\label{TableSphDim3Dis4}
\normalsize
\end{table}
In \cite[Theorem~3.8]{NM} it was proved, amongst other results, that the maximum cardinality of spherical $3$-distance sets in $\mathbb{R}^4$ is at most $27$. It turns out that the exact answer, presented below, is considerably smaller.
\begin{theorem}\label{3disr4}
The largest cardinality of a spherical $3$-distance set in $\mathbb{R}^4$ is $13$. There are exactly four configurations realizing this up to isometry.
\end{theorem}
In what follows we discuss the configurations arising in Theorem~\ref{3disr4}. The largest candidate Gram matrices found by the search are:
\setlength{\arraycolsep}{2.5pt}%
\[S_{13A}(u,v,w):=\left[\begin{smallmatrix}
S_{12A}(u,v,w) & w J\\
w J^T & 1\\
\end{smallmatrix}\right],\qquad S_{13B}(u,v,w):=\left[\begin{smallmatrix}
1 & u J^T\\ 
u J & S_{12A}(u,v,w)\\
\end{smallmatrix}\right],\quad\text{and}\]
\[S_{13C}(u,v,w):=\left[
\begin{smallmatrix}
 1 & u & u & u & u & v & v & v & v & w & w & w & w \\
 u & 1 & v & v & w & u & u & v & w & u & v & w & w \\
 u & v & 1 & w & v & u & w & u & v & v & w & u & w \\
 u & v & w & 1 & v & v & u & w & u & w & u & w & v \\
 u & w & v & v & 1 & w & v & u & u & w & w & v & u \\
 v & u & u & v & w & 1 & w & w & u & v & u & v & w \\
 v & u & w & u & v & w & 1 & u & w & v & v & w & u \\
 v & v & u & w & u & w & u & 1 & w & u & w & v & v \\
 v & w & v & u & u & u & w & w & 1 & w & v & u & v \\
 w & u & v & w & w & v & v & u & w & 1 & u & u & v \\
 w & v & w & u & w & u & v & w & v & u & 1 & v & u \\
 w & w & u & w & v & v & w & v & u & u & v & 1 & u \\
 w & w & w & v & u & w & u & v & v & v & u & u & 1 \\
\end{smallmatrix}
\right], \qquad J:=(1,1,\dots,1)^T\in\mathbb{R}^{12},\]
where $S_{12A}(u,v,w)$ is the first matrix shown in \eqref{icosahere}. Once again, we solve the system of equations \eqref{main_eq2} coming from Proposition~\ref{prop21} to ascertain that the rank of these matrices is $4$. In the first case we have $u=(5\pm3\sqrt{5})/20$, $v=1/2-u$, $w=-1/2$. These two algebraic solutions are permutation equivalent. In the second case we have $u=(1\pm\sqrt{5})/4$, $v=1/2$ $w=u-1/2$. These two are nonisometric solutions, since the corresponding Gram matrices have different spectrum. In the third case we have $64 u^3+16 u^2-16 u+1=0$, and $v=(16 u^2+4 u-3)/4$, $w=(-8 u^2-4 u+1)/2$. Let $\alpha<\beta<\gamma$ denote the three real roots of the polynomial $64 u^3+16 u^2-16 u+1$. Then $G_{13C}(\alpha,\gamma,\beta)$, $G_{13C}(\beta,\alpha,\gamma)$, and $G_{13C}(\gamma,\beta,\alpha)$ are the three permutation equivalent solutions. Each of these four matrices are positive semidefinite, as required. Refer to Table~\ref{TableSphDim3Dis4} for the number of intermediate objects found. This concludes the discussion of Theorem~\ref{3disr4}.
\begin{table}[htbp]%
\tiny
\begin{tabular}{ccccccccccc}
\hline
$n$                                                   &   5 &     6 &     7 &   8 &   9 &  10 & 11 & 12 & 13 & 14\\
\hline
\#Candidate Gram matrices                             & 142 & 4300 & 205646 & 891 & 396 & 173 & 62 & 19 & 3 & 0\\
\#Candidate Gram matrices satisfying \eqref{main_eq2} & 142 & 3816 &   1748 & 889 & 396\\
\hline
\end{tabular}
\caption{Spherical $n$-element $3$-distance sets in $\mathbb{R}^4$}
\label{TableSphDim4Dis3}
\normalsize
\end{table}
\begin{rem}\label{rem1}
Let $I_{13}$ be the identity matrix of order $13$, and let $M(u,v,w):=S_{13C}(u,v,w)-I_{13}$. Then one may consider the union of the set of columns of the matrices $M(1,0,0)$, $M(0,1,0)$, $M(0,0,1)$, and then concatenate these to form a $(0,1)$-matrix of size $13\times 39$: a near-resolvable design $\mathrm{NRB}(13,4,3)$, see \cite[Chapter~7.2]{CRC} and \cite{GHK}.
\end{rem}
\begin{rem}\label{rem2}
Let $J_{13}\in\mathbb{R}^{13}$ be the column vector with all entries $1$, and let $\zeta$ be a primitive complex third root of unity. Then the matrix $C:=\left[\begin{smallmatrix} 0 & J_{13}^T\\ J_{13} & S_{13C}(1,\zeta,\zeta^2)-I_{13}\end{smallmatrix}\right]$ is a symmetric generalized conference matrix $\mathrm{GC}(\mathcal{C}_3;4)$ over the (multiplicatively written) cyclic group $\mathcal{C}_3=\{1,\zeta,\zeta^2\}$, see \cite[Chapter~6.2]{CRC}.
\end{rem}
It is most fascinating that the matrix $S_{13C}(u,v,w)$ is related to a number of well-known combinatorial objects (see Remarks~\ref{rem1} and \ref{rem2} above). We conclude the discussion of spherical few-distance sets with a remark on the next open case.
\begin{rem}
The vertices of the $24$-cell, given by the set $\{\sigma(\pm1,\pm1,0,0)^T\colon \sigma\in\mathfrak{S}_4\}$ is a spherical $4$-distance set in $\mathbb{R}^4$. We do not know whether this is the best possible.
\end{rem}
\section{General few-distance sets}
Now we turn to the discussion of the general case, where the elements of $\mathcal{X}$ are not necessarily unit vectors. The following classical result is presented in various equivalent forms within the cited references. Here we recall it in a form which is the most useful for our purposes. The proof is once again included for the reader's convenience.
\begin{theorem}[\mbox{\cite[Theorem~2.2]{EDG}, \cite[Theorem~7.1]{L}, \cite{M}, \cite{NEU}, \cite{S}}]\label{t31}
Let $d\geq 1$, $n\geq 2$ be integers. There exists $n$ distinct vectors $v_1,\dots,v_n$ in $\mathbb{R}^d$ with mutual distances $\mu(v_i,v_j)$, $i,j\in\{1,\dots,n\}$, if and only if the matrix $C:=[\mu(v_i,v_n)^2+\mu(v_j,v_n)^2-\mu(v_i,v_j)^2]_{i,j=1}^{n-1}$ is positive semidefinite, $\rank C\leq d$, $C_{ii}>0$ for every $i\in\{1,\dots,n-1\}$, and $C_{ij}<(C_{ii}+C_{jj})/2$ for every $i<j$ with $i,j\in\{1,\dots,n-1\}$.
\end{theorem}
\begin{proof}
Assume that we have a set of $n$ distinct vectors $v_1$, $\dots$, $v_n$ with pairwise distances $\mu(v_i,v_j)>0$ for every $i<j$ and $i,j\in\{1,\dots,n\}$. By the polarization identity, we have $\mu(v_i,v_n)^2+\mu(v_j,v_n)^2-\mu(v_i,v_j)^2=2\left\langle v_i-v_n,v_j-v_n\right\rangle$ and therefore $C$ is a Gram matrix. Once again, $C$ is positive semidefinite, and $\rank C$ is the maximum number of linearly independent vectors amongst $v_i-v_n$, $i\in\{1,\dots, n-1\}$. Therefore $\rank C\leq d$, as claimed. Finally, the conditions $C_{ii}>0$ and $C_{ij}<(C_{ii}+C_{jj})/2$ follow as the vectors are distinct, and therefore $\mu(v_i,v_j)>0$ for every $i<j$ with $i,j\in\{1,\dots,n\}$.

Conversely, let $C$ be an $(n-1)\times (n-1)$ positive semidefinite matrix with the conditions stated. Then one may define a set of $\binom{n}{2}$ positive real numbers $\mu_{in}:=\sqrt{C_{ii}/2}$, $i\in\{1,\dots,n-1\}$, and $\mu_{ij}:=\sqrt{C_{ii}/2+C_{jj}/2-C_{ij}}$ for every $i<j$ with $i,j\in\{1,\dots,n-1\}$, and reconstruct $n-1$ vectors by the Cholesky decomposition with complete pivoting. The procedure will result in a $\rank C\times (n-1)$ matrix $V$ such that $V^TV=C$. Since $\mu_{ij}>0$, the column vectors of $V$ are distinct, which, together with $v_n:=(0,\dots,0)\in\mathbb{R}^{\rank C}$, will form the vectors $v_i$, $i\in\{1,\dots,n\}$ having the prescribed distances $\mu(v_i,v_j)=\mu_{ij}$ for every $i<j$ with $i,j\in\{1,\dots,n\}$. This reconstruction is unique up to isometry \cite{NJH}.
\end{proof}
\begin{rem}
In $\mathbb{R}^d$ the union of an $n$-element spherical $s$-distance set and its center point (see Section~\ref{section2}) forms an $(n+1)$-element general $t$-distance set for some $t\in\{s,s+1\}$.
\end{rem}
Recall that in Section~\ref{section2} we have classified what we called the candidate Gram matrices having constant diagonal $1$ and $s$ distinct off-diagonal elements representing the combinatorial properties of a distance set. These objects will be used once again during the treatment of the general case. The analogue statement to Lemma~\ref{lemma1} also applies here, and the analogue of Proposition~\ref{prop21} is the following. One key difference compared to the spherical case is that here one of the distances can be normalized to $1$ up to a global isometry.
\begin{proposition}\label{prop31}
Let $d,s\geq 1$, $n\geq d+2$ be integers, let $x_1$, $\dots$, $x_s$ be indeterminates, and let $G(x_1,\dots,x_s)$ be an $n\times n$ candidate Gram matrix. Let $\mathcal{M}$ denote the set of all $(d+1)\times (d+1)$ submatrices of the matrix $C:=[G_{in}+G_{jn}-G_{ij}+\delta_{ij}]_{i,j=1}^{n-1}$, where $\delta_{ij}$ is the Kronecker symbol. There exists $s-1$ distinct complex numbers $\alpha_2$, $\dots$, $\alpha_s$, each different from $0$ and $1$, so that $\rank C(1,\alpha_2,\dots,\alpha_s)\leq d$ if and only if the following system of $\binom{n-1}{d+1}^2+1$ polynomial equations in $s$ variables
\begin{equation}\label{main_eq3}
    \mathrm{det}M(1,\dots,x_s) = 0,\quad \text{ for all } M(x_1,\dots,x_s)\in\mathcal{M},\qquad 1+u\prod_{i=2}^sx_i(x_i-1)\prod_{2\leq j<k\leq s}(x_j-x_k)=0
\end{equation}
has a complex solution.
\end{proposition}
\begin{proof}
This is once again just a reformulation of the rank condition in terms of vanishing minors. Moreover, it is assumed that $x_1=1$ up to a global isometry. The auxiliary variable $u$ ensures that the other variables take up distinct values, each different from $0$ and $1$.
\end{proof}
\begin{rem}
Using the notation of Proposition~\ref{prop31}, one may observe that the determinantal ideal \cite{ALDO} generated by the polynomial equations $\mathrm{det}M(x_1,\dots,x_s)=0$, $M(x_1,\dots,x_s)\in\mathcal{M}$ is homogenous of degree $d+1$. Therefore, if these equations have a common solution $(x_1,\dots,x_s)=(\alpha_1,\dots,\alpha_s)$, where $\prod_{i=1}^s\alpha_i\neq 0$, then $(\alpha_1/\alpha_t,\alpha_2/\alpha_t,\dots,\alpha_s/\alpha_t)$ is also a solution for any $t\in\{1,\dots, s\}$. This justifies the normalization of one of the variables.
\end{rem}
The search in the general case is analogous to what is described in Section~\ref{section23}, with the only difference that it starts with the $(d+2)\times (d+2)$ candidate Gram matrices, and then it uses Proposition~\ref{prop31} for pruning. We omit the details.

\subsection{Results} Now we turn to the discussion of the following new classification results.
\begin{theorem}\label{t34}
The largest cardinality of a $4$-distance set in $\mathbb{R}^3$ is $13$. There are exactly two configurations realizing this up to isometry: the vertices of the icosahedron with its center point, and the vertices of the cuboctahedron with its center point.
\end{theorem}
The search revealed two configurations, which can be described with the following candidate Gram matrices:
\setlength{\arraycolsep}{2.5pt}%
\[G_{13A}(u,v,w,x):=\left[\begin{smallmatrix}
S_{12A}(u,v,w) & x J\\
x J^T & 1\\
\end{smallmatrix}\right],\quad G_{13B}(u,v,w,x):=\left[\begin{smallmatrix}
1 & u J^T\\ 
u J & S_{12C}(u,v,w,x)\\
\end{smallmatrix}\right],\]
where the submatrices $S_{12A}(u,v,w)$ and $S_{12C}(u,v,w,x)$ were discussed in Section~\ref{sect34}, see \eqref{icosahere}. Recall that the rank condition applies to some $12\times 12$ transformed matrices as stated in Theorem~\ref{t31}. We assume that $x:=1$ due to a global isometry. In the first case the rank condition implies that $u=2(5\pm\sqrt{5})/5$, $v=4-u$, $w=4$. These two algebraic solutions are permutation equivalent. In the second case we find that $u=(2\pm1)/4$, $v=1/2$, $w=1-u$. These two solutions are also permutation equivalent. Thus, altogether two nonisometric solutions are found. It is not too hard to see that these correspond to the claimed sets in Theorem~\ref{t34}. Refer to Table~\ref{TableGenDim3Dis4} for the number of intermediate objects found.
\begin{table}[htbp]%
\tiny
\begin{tabular}{cccccccccccc}
\hline
$n$                                                   &  4 &   5 &     6 &        7 &    8 &   9 & 10 & 11 & 12 & 13 & 14\\
\hline
\#Candidate Gram matrices                             & 22 & 513 & 67669 & 24617591 & 1093 & 277 & 59 & 12 & 5 & 2 & 0\\
\#Candidate Gram matrices satisfying \eqref{main_eq3} & 22 & 512 & 62095 &     4499 & 1093\\
\hline
\end{tabular}
\caption{General $n$-element $4$-distance sets in $\mathbb{R}^3$}
\label{TableGenDim3Dis4}
\normalsize
\end{table}

We regard the following as the main result of this manuscript.
\begin{theorem}\label{mainthmxxx}
The largest cardinality of a $3$-distance set in $\mathbb{R}^4$ is $16$. There is a unique configuration realizing this up to isometry.
\end{theorem}
The following is the largest candidate Gram matrix found by the search:
\[G_{16}(u,v,w):=\left[
\begin{smallmatrix}
1 & u & u & u & u & u & u & u & u & u & u & u & u & v & v & v \\
 u & 1 & u & u & u & u & u & u & u & u & v & v & v & u & u & u \\
 u & u & 1 & u & u & u & v & v & v & w & u & u & u & u & u & u \\
 u & u & u & 1 & v & v & u & u & w & v & u & u & w & u & u & w \\
 u & u & u & v & 1 & v & u & w & u & v & u & w & u & u & w & u \\
 u & u & u & v & v & 1 & w & u & u & v & w & u & u & w & u & u \\
 u & u & v & u & u & w & 1 & v & v & u & u & w & w & u & w & w \\
 u & u & v & u & w & u & v & 1 & v & u & w & u & w & w & u & w \\
 u & u & v & w & u & u & v & v & 1 & u & w & w & u & w & w & u \\
 u & u & w & v & v & v & u & u & u & 1 & w & w & w & w & w & w \\
 u & v & u & u & u & w & u & w & w & w & 1 & v & v & u & w & w \\
 u & v & u & u & w & u & w & u & w & w & v & 1 & v & w & u & w \\
 u & v & u & w & u & u & w & w & u & w & v & v & 1 & w & w & u \\
 v & u & u & u & u & w & u & w & w & w & u & w & w & 1 & v & v \\
 v & u & u & u & w & u & w & u & w & w & w & u & w & v & 1 & v \\
 v & u & u & w & u & u & w & w & u & w & w & w & u & v & v & 1 \\
\end{smallmatrix}
\right].\]
Recall that the rank condition applies to the $15\times 15$ transformed matrix $C$ as stated in Theorem~\ref{t31}. One readily verifies that the only possibility to have rank $4$ with $u:=1$ normalization is the choice $v=2$ and $w=3$. In this case the eigenvalues of the matrix $C$ are $4$, and the three roots of $\lambda^3-52\lambda^2+500\lambda-1312$. Since the coefficients of this polynomial alternate in sign, clearly its roots are positive (recall that $C$ is symmetric), and therefore $C$ is positive semidefinite. Table~\ref{TableGenDim4Dis3} displays the number of intermediate objects found during the search. This concludes the discussion of Theorem~\ref{mainthmxxx}.

\begin{table}[htbp]%
\tiny
\begin{tabular}{ccccccccccccc}
\hline
$n$                                                   &     6 &      7 &    8 &    9 &   10 &   11 & 12 & 13 & 14 & 15 & 16 & 17\\
\hline
\#Candidate Gram matrices                             &  4300 & 384183 & 6939 & 2496 & 1473 & 765 & 341 & 113 & 31 & 8 & 1 & 0\\
\#Candidate Gram matrices satisfying \eqref{main_eq3} &  4299 &  16481 & 5043 & 2496 &\\
\hline
\end{tabular}
\caption{General $n$-element $3$-distance sets in $\mathbb{R}^4$.}
\label{TableGenDim4Dis3}
\normalsize
\end{table}
\begin{rem}
The upper left $10\times 10$ submatrix of $G_{16}(1,2,3)$ represents a $3$-dimensional cube, and a bipyramid over it in the $4$-dimensional space. This structural observation motivated the choice of the set of distances $A$ in Theorem~\ref{thm41}.
\end{rem}
Finally, we independently obtain the following result from the literature.
\begin{theorem}[\cite{Shinohara},\mbox{\cite[Theorem~1.2]{Sico}}]
The largest cardinality of a $3$-distance set in $\mathbb{R}^3$ is $12$. There is a unique configuration realizing this up to isometry: the vertices of the icosahedron.
\end{theorem}
The candidate Gram matrix found by the search is $S_{12A}(u,v,w)$, as expected. We find that the system of equations \eqref{main_eq3} has the following normalized solutions: $u=(5\pm\sqrt{5})/10$, $v=(5\mp\sqrt{5})/10$, $w=1$. These two solutions are permutation equivalent. Refer to Table~\ref{TableGenDim3Dis3} for the number of intermediate objects generated.
\begin{table}[htbp]%
\tiny
\begin{tabular}{ccccccccccc}
\hline
$n$                                                   &  4 &   5 &    6 &   7 &  8 & 9 & 10 & 11 & 12 & 13\\
\hline
\#Candidate Gram matrices                             & 15 & 142 & 4288 & 106 & 19 & 5 &  2 &  1 &  1 &  0\\
\#Candidate Gram matrices satisfying \eqref{main_eq3} & 15 & 141 &  434 &  90 & 19 &\\
\hline
\end{tabular}
\caption{General $n$-element $3$-distance sets in $\mathbb{R}^3$}
\label{TableGenDim3Dis3}
\normalsize
\end{table}

The following result settles \cite[Conjecture~2]{EF}. See also \cite[Theorem~19]{xx} for previous partial results.
\begin{theorem}\label{erdosthm}
The largest cardinality of a $6$-distance set in $\mathbb{R}^2$ is $13$. There are exactly three configurations realizing this up to isometry: the regular convex $13$-gon, the regular convex $12$-gon with its center point, and a regular hexagram with its center point \textup{(}as shown on \cite[p.~116]{EF}\textup{)}.
\end{theorem}

The search found the following three candidate Gram matrices $G_{13C}(u,v,w,x,y,z)$, $G_{13D}(u,v,w,x,y,z)$, and $G_{13E}(u,v,w,x,y,z)$:
\[\left[
\begin{smallmatrix}
  1 & u & u & u & u & u & u & u & u & u & u & u & u \\
 u & 1 & u & u & v & v & w & x & x & y & y & z & z \\
 u & u & 1 & v & u & w & v & x & y & x & z & y & z \\
 u & u & v & 1 & w & u & v & y & x & z & x & z & y \\
 u & v & u & w & 1 & v & u & y & z & x & z & x & y \\
 u & v & w & u & v & 1 & u & z & y & z & x & y & x \\
 u & w & v & v & u & u & 1 & z & z & y & y & x & x \\
 u & x & x & y & y & z & z & 1 & u & u & v & v & w \\
 u & x & y & x & z & y & z & u & 1 & v & u & w & v \\
 u & y & x & z & x & z & y & u & v & 1 & w & u & v \\
 u & y & z & x & z & x & y & v & u & w & 1 & v & u \\
 u & z & y & z & x & y & x & v & w & u & v & 1 & u \\
 u & z & z & y & y & x & x & w & v & v & u & u & 1 \\
\end{smallmatrix}
\right],\left[\begin{smallmatrix}
 1 & u & u & v & v & w & w & x & x & y & y & z & z \\
 u & 1 & v & u & w & v & x & w & y & x & z & y & z \\
 u & v & 1 & w & u & x & v & y & w & z & x & z & y \\
 v & u & w & 1 & x & u & y & v & z & w & z & x & y \\
 v & w & u & x & 1 & y & u & z & v & z & w & y & x \\
 w & v & x & u & y & 1 & z & u & z & v & y & w & x \\
 w & x & v & y & u & z & 1 & z & u & y & v & x & w \\
 x & w & y & v & z & u & z & 1 & y & u & x & v & w \\
 x & y & w & z & v & z & u & y & 1 & x & u & w & v \\
 y & x & z & w & z & v & y & u & x & 1 & w & u & v \\
 y & z & x & z & w & y & v & x & u & w & 1 & v & u \\
 z & y & z & x & y & w & x & v & w & u & v & 1 & u \\
 z & z & y & y & x & x & w & w & v & v & u & u & 1 \\
\end{smallmatrix}\right],
\left[
\begin{smallmatrix}
 1 & u & u & u & u & u & u & v & v & v & v & v & v \\
 u & 1 & u & u & v & v & w & u & u & w & w & x & x \\
 u & u & 1 & v & u & w & v & u & w & u & x & w & x \\
 u & u & v & 1 & w & u & v & w & u & x & u & x & w \\
 u & v & u & w & 1 & v & u & w & x & u & x & u & w \\
 u & v & w & u & v & 1 & u & x & w & x & u & w & u \\
 u & w & v & v & u & u & 1 & x & x & w & w & u & u \\
 v & u & u & w & w & x & x & 1 & v & v & y & y & z \\
 v & u & w & u & x & w & x & v & 1 & y & v & z & y \\
 v & w & u & x & u & x & w & v & y & 1 & z & v & y \\
 v & w & x & u & x & u & w & y & v & z & 1 & y & v \\
 v & x & w & x & u & w & u & y & z & v & y & 1 & v \\
 v & x & x & w & w & u & u & z & y & y & v & v & 1 \\
\end{smallmatrix}
\right].\]
We use the normalization $u=1$ up to a global isometry. In the first case we find that $v=3$, $w=4$, $x=2\pm\sqrt{3}$, $y=2$, $z=4-x$ are the only solutions corresponding to a planar configuration. The two solutions are permutation equivalent, thus represent isometric configurations: the regular convex $12$-gon, and its center point. In the second case, after discarding the meaningless complex solutions satisfying $v^2-v+1=0$, we get $v^6-11 v^5+45 v^4-84 v^3+70 v^2-21 v+1=0$, which is the minimal polynomial of $\sin^2 (\pi/13)/\sin^2(6\pi/13)$. Any choice for $v$ then uniquely determines the remaining parameter values as follows: $w=v^2-2 v+1$, $x=v^3-4 v^2+4 v$, $y=v^4-6 v^3+11 v^2-6 v+1$, $z=v^5-8 v^4+22 v^3-24 v^2+9 v$. All these solutions correspond to the regular convex $13$-gon. In the third case we get $v=3$, $w=4$, $x=7$, $y=9$, $z=12$, which corresponds to the hexagram. Refer to Table~\ref{TableGenDim2Dis6} for the number of intermediate objects found. We note that due to the (relatively) large number of distinct distances the classification of this case required considerable amount of computational efforts. This concludes the discussion of Theorem~\ref{erdosthm}.
\begin{table}[htbp]%
\tiny
\begin{tabular}{cccccccccccc}
\hline
$n$                                                   &  4 &    5 &      6 &     7 &    8 &   9 &  10 & 11 & 12 & 13 & 14\\
\hline
\#Candidate Gram matrices                             & 25 & 1193 & 537230 & 14732 & 1565 & 635 & 228 & 62 & 15 & 3 & 0\\
\#Candidate Gram matrices satisfying \eqref{main_eq3} & 24 &  922 &  20229 &  4667 & 1565 &\\
\hline
\end{tabular}
\caption{General $n$-element $6$-distance sets in $\mathbb{R}^2$}
\label{TableGenDim2Dis6}
\normalsize
\end{table}

We firmly believe that the approach outlined in this paper has enormous further potential. In particular, the classification of maximum $3$-distance sets in $\mathbb{R}^5$, and maximum planar $7$-distance sets will be doable in the near-term future.
\section*{Acknowledgements}
Research on this topic began during the ``Tight frames and Approximation'' workshop, 20--23 February, 2018 in Taipa, New Zealand. F.~Sz.~is grateful for Prof.~Shayne Waldron for his amazing hospitality.

\end{document}